\documentclass[11pt, a4paper]{amsart}
\usepackage{amsfonts}
\usepackage{amssymb,amsthm, amsmath}
\usepackage{bbm}
\usepackage[vmargin=2cm,hmargin=2cm]{geometry}
\usepackage[utf8]{inputenc}
\usepackage{xcolor}
\usepackage{pgf,tikz}
\usepackage{url}
\usepackage{euler}

\newtheorem{theorem}{Theorem}

\newtheorem{proposition}[theorem]{Proposition}
\newtheorem{corollary}[theorem]{Corollary}
\newtheorem{lemma}[theorem]{Lemma}

\theoremstyle{remark}

\newtheorem{example}[theorem]{Example}

\newtheorem{remark}[theorem]{Remark}

\newtheorem{notation}[theorem]{Notation}

\title{Factorization invariants in half-factorial affine semigroups}

\author{{P.A.} {Garc\'{\i}a S\'anchez}}
\address{Departamento de \'Algebra, Universidad de Granada, E-18071 Granada, Espa\~na}
\email{pedro@ugr.es}

\author{I. Ojeda}
\address{Departamento de Matem\'{a}ticas, Universidad de Extremadura,  E-06071 Badajoz, Espa\~na}
\email{ojedamc@unex.es}

\author{{A.} S\'anchez-R.-Navarro}
\address{Departamento Lenguajes y Sistemas Inform\'aticos, Universidad de C\'adiz, E-11405 Jerez
de la Frontera (C\'adiz),  Espa\~na }
\email{alfredo.sanchez@uca.es}

\thanks{The first author is supported by the projects MTM2010-15595 and FQM-343,  FQM-5849, and FEDER funds. The second author is supported by the project MTM2007-64704, National Plan I+D+I and by Junta de Extremadura (FEDER funds). The third author is partially supported by Junta de Andaluc\'{\i}a group FQM-366}


\begin{document}

\begin{abstract}
Let $\mathbb{N} \mathcal{A}$ be the monoid generated by $\mathcal{A} = \{\mathbf{a}_1, \ldots, \mathbf{a}_n\} \subseteq \mathbb{Z}^d.$ We introduce the homogeneous catenary degree of $\mathbb{N} \mathcal{A}$ as the smallest $N \in \mathbb N$ with the following property: for each $\mathbf{a} \in \mathbb{N} \mathcal{A}$ and any two factorizations $\mathbf{u}, \mathbf{v}$ of $\mathbf{a}$, there exists factorizations $\mathbf{u} = \mathbf{w}_1, \ldots, \mathbf{w}_t = \mathbf{v} $ of $\mathbf{a}$ such that, for every $k,\ \mathrm{d}( \mathbf{w}_k, \mathbf{w}_{k+1}) \leq N,$ where $\mathrm{d}$ is the usual distance between factorizations, and the length of $\mathbf{w}_k,\ \vert \mathbf{w}_k \vert,$ is less than or equal to $\max\{\vert \mathbf{u} \vert, \vert \mathbf{v} \vert \}.$ We prove that the homogeneous catenary degree of $\mathbb{N} \mathcal{A}$  improves the monotone catenary degree as upper bound for the ordinary catenary degree, and we show that it can be effectively computed. We also prove that for half-factorial monoids, the tame degree and the $\omega$-primality coincide, and that all possible catenary degrees of the elements of an affine semigroup of this kind occur as the catenary degree of one of its Betti elements.
\end{abstract}

\maketitle
\section*{Introduction}

Non-unique factorization invariants can be divided into two big groups: the first contains those based on the lengths of the factorizations of an element, whilst the second exploits the idea of distance between factorizations. In a half-factorial monoid, all factorizations of a given element have the same length, and thus the first group simply characterizes the half-factorial property without giving extra information. This is why we will focus on the second group containing the catenary and tame degree. The $\omega$-primality is not in any of these two groups, but as we prove in the last section, in the half-factorial setting, it coincides with the tame degree.

An element in a cancellative monoid might be expressed in different ways as a linear combination with nonnegative integer coefficients of its generators. Each such expression is usually known as a factorization of the element. The distance between two different factorizations is the largest length (number of generators) of the factorizations resulting after removing their common part. Even if two factorizations of a given element are too far away one from the other, it may happen that we can join them by a chain of factorizations with the property that the distance of two consecutive elements in the chain are bounded by a fixed amount. The least possible of these bounds is the catenary degree of the element, and the supremum of all catenary degrees of all the elements in the monoid is the catenary degree of the monoid itself. 

It was shown in \cite{Chapman} that the catenary degree can be computed by using certain (minimal) presentations of the monoid, and thus according to \cite{herzog} these computations can also be made by using binomial ideals. 

We prove that if an affine semigroup is half-factorial, then any catenary degree of any element in the monoid is the catenary degree of one of its Betti elements. As a consequence, its catenary degree can be computed as the maximal total degree of a minimal generating system of its associated binomial ideal (which is homogeneous). 

From an affine semigroup $S$ we construct $S^\mathsf{eq}$ and $S^\mathsf{hom}$. These two monoids are half-factorial, and the catenary degrees of $S^\mathsf{eq}$ and $S^\mathsf{hom}$ are upper bounds of the catenary degree of $S$. The first one corresponds to the well known equal catenary degree (the lengths of the factorizations in the chains are equal), while the second is a lower bound of the monotone catenary degree (the lengths are non-decreasing), which we call homogeneous catenary degree. Both equal and monote catenary degrees can be computed by using linear integer programming (see \cite{char-cmon-2}). The advantage of using binomial ideals is that the concept of catenary degree translates to that of total degree, and thus computation of equal and homogeneous catenary degrees can be done by looking at the largest total degree of minimal systems of generators of two binomial ideals. Hence, instead of integer programming one can use Gr\"obner basis computations and with the help of any computer algebra system, these two catenary degrees can be easily calculated. Indeed, all computations in the examples we give (and the experiments that led to our results) have been performed with the \texttt{numericalsgs GAP} package (\cite{numericalsgps}). Moreover, bounds for minimal generators of homogeneus toric ideals can be used to find upper bounds for the catenary degree of $S$.

The tame degree of an atomic monoid is the minimum $N$ such that for any factorization of an element in the monoid and any atom dividing this element, there exists another factorization at distance at most $N$ from the original factorization in which this atom occurs (in a cancellative monoid, when using additive notation, $a$ divides $b$ means that $b-a$ is in the monoid).  

The $\omega$-primality measures how far the irreducibles of a monoid are from being prime: it is the minimum $N$ such that whenever an irreducible element divides a sum of elements, then it divides a subsum with at most $N$ elements. We show in the last section that, tame degree and $\omega$-primality coincide for half-factorial affine semigroups.

\section{Presentations, binomial ideals, and factorizations}

Let $\mathbbmss{k}[X] :=  \mathbbmss{k}[X_1, \ldots, X_n]$ be the polynomial ring in $n$ variables over a field $\mathbbmss{k}.$ As usual, we will write $X^\mathbf{u}$ for the monomial $X_1^{u_1} \cdots X_n^{u_n} \in \mathbbmss{k}[X]$, and will define the \textbf{degree} of a monomial $X^\mathbf{u} \in \mathbbmss{k}[X]$ as $\deg(X^\mathbf{u}) = \sum_{i=1}^n u_i$. 

Let $\mathcal{A} = \{\mathbf{a}_1, \ldots, \mathbf{a}_n\} \subseteq \mathbb{Z}^d$, and let $A$ be the matrix whose rows are $\mathbf a_1,\ldots,\mathbf a_n$. The semigroup homomorphism 
$$
\begin{array}{rcl}
\pi : \mathbb{N}^n & \longrightarrow & \mathbb{N} \mathcal{A} := \mathbb{N} \mathbf{a}_1 + \cdots + \mathbb{N} \mathbf{a}_n\\ \mathbf{u} = (u_1, \ldots, u_n) & \longmapsto & \mathbf{u} A := \sum_{i=1}^n u_i \mathbf{a}_i
\end{array}
$$
defines a homomorphism of semigroup algebras $$\tilde \pi : \mathbbmss{k}[X] \longrightarrow \mathbbmss{k}[\mathcal{A}] := \bigoplus_{\mathbf{a} \in \mathbb{N} \mathcal{A}} \mathbbmss{k} \chi^\mathbf{a};\ X^\mathbf{u} \longmapsto \chi^{\mathbf{u} A}.$$ The kernel of $\tilde \pi$ is the toric ideal $$I_\mathcal{A} = \big\langle \{ X^\mathbf{u} - X^\mathbf{v}\ \mid\ \mathbf{u}, \mathbf{v} \in \mathbb{N}^n\ \text{with}\  \pi(\mathbf{u}) = \pi(\mathbf{v})\} \big\rangle$$ (see, e.g. \cite[Lema 4.1]{Sturmfels}). Thus from the kernel of $\tilde \pi$ we can construct a presentation for $\mathbb N\mathcal A$, that is, a system of generators for the congruence $\ker \pi=\{ (\mathbf u,\mathbf v)\in \mathbb N^n\times \mathbb N^n ~|~ \pi(\mathbf u)=\pi(\mathbf v)\}$ (see \cite{herzog}).

Given $\mathbf a\in \mathbb N\mathcal A$, the set $\mathsf Z(\mathbf a)=\pi^{-1}(\mathbf a)$ is the \textbf{set of factorizations} of $\mathbf a$. For a factorization $\mathbf u=(u_1,\ldots,u_n)$  of an element $\mathbf a \in \mathbb N\mathcal A$, its \textbf{length} is $|\mathbf u|=\sum_{i=1}^n u_i=\deg(X^\mathbf u)$. The \textbf{set of lengths} of $\mathbf a$ is $\mathsf L(\mathbf a)=\{ |\mathbf u| ~|~ \mathbf u\in \mathsf Z(\mathbf a)\}$.

\begin{remark}
We are going to assume that $\mathbb N\mathcal A$ is reduced, that is, $(-\mathbb N\mathcal A)\cap \mathbb N\mathcal A =\{0\}$, or equivalently, its only unit is the zero element. This restriction is motivated by two facts: the first is that units are not considered as parts of factorizations, and the second is that in the reduced case, in order to find a minimal system of generators for $I_\mathcal A$ we only have to look at certain non-connected graphs that we will define below.

In this setting, the sets $\mathsf Z(\mathbf a)$ are always finite (indeed all its elements are incomparable with respect to the usual partial ordering on $\mathbb N^n$; and the finiteness follows from Dickson's lemma or by Gordan's lemma, see for instance \cite{libro-fg}).

We are not assuming that $\mathcal A$ is the minimal system of generators of $\mathbb N\mathcal A$. It is well known (see for instance \cite[Exercise 6, Chapter 3]{libro-fg}) that $\mathbb N\mathcal A$ admits a unique minimal system of generators, and its elements are precisely those that cannot be expressed as sums of two other non-unit elements in $\mathbb N\mathcal A$. These elements are usually called \textbf{irreducibles} or \textbf{atoms} of the monoid.
\end{remark}

\subsection*{Half-factorial monoids}

The monoid $\mathbb N\mathcal A$ is \textbf{half-factorial} if for every  $\mathbf a\in \mathbb N\mathcal A$, $\sharp \mathsf L(\mathbf a)=1$, that is, the lengths of all factorizations of $\mathbf a$ are equal. This means that the ideal $I_\mathcal A$ is homogeneous. In view of \cite[Lemma 4.14]{Sturmfels}, there exists $\omega \in \mathbb{Q}^d$ such that $$A\, \omega^T=(1,\ldots,1)^T,$$ or  equivalently, $\mathbf a_i\cdot \omega=1$, for all $i\in\{1,\ldots, n\}$, where the dot product is defined as usual: $(x_1,\ldots,x_n)\cdot (y_1,\ldots,y_n)=x_1y_1+\cdots+ x_ny_n$. Indeed, the converse is also true, because if there exists such an $\omega$, then for any two factorizations $\mathbf u, \mathbf v$ of an element $\mathbf a\in\mathbb N\mathcal A$, $\mathbf  a=\pi(\mathbf u)=\pi(\mathbf v)$, and thus $\mathbf a=\mathbf u A=\mathbf v A$. Hence $\mathbf a \cdot \omega =\mathbf u A\, \omega^T = \mathbf v A\, \omega^T$, which leads  to $\mathbf a \cdot \omega=\mathbf u \cdot (1,\ldots,1) = \mathbf v \cdot (1,\ldots,1)$, that is, $\mathbf a \cdot \omega=|\mathbf u|=|\mathbf v|$.  In particular, we have shown the following, which will be used later.

\begin{lemma}\label{longitudes-half-factorial}
The monoid $\mathbb N\mathcal A$ is half-factorial if and only if there exists $\omega\in\mathbb Q^d$ such that $A\, \omega^T=(1,\ldots,1)^T$. If this is the case, $\mathsf L(\mathbf a)=\{\mathbf a \cdot \omega\},$ for every $\mathbf a\in\mathbb N\mathcal A$.
\end{lemma}

We define the $\mathcal{A}-$\textbf{degree} of a monomial $X^\mathbf{u} \in \mathbbmss{k}[X]$ as $\deg_\mathcal{A} (X^\mathbf{u}) = \sum_{i=1}^n u_i \mathbf{a}_i$ ($=\pi(\mathbf{u}))$. 

So, we have that  if $\mathbb N\mathcal A$ is half-factorial, with $\omega\in\mathbb Q$ such that $A\, \omega^T=(1,\ldots,1)^T$, then  \begin{equation}\label{cat_ecu2}\deg_\mathcal{A}(X^\mathbf{u})  = \deg_\mathcal{A}(X^\mathbf{v}) = \mathbf{a} \hbox{ implies } \deg(X^\mathbf{u})  = \deg(X^\mathbf{v}) =  \mathbf{a} \cdot \omega.\end{equation} 

\subsection*{Betti elements}

Let $M_\mathbf{a}=\{X^\mathbf{u} ~|~ \mathbf u\in \mathsf Z(\mathbf a)\}$ be the set of monomials of $\mathbbmss{k}[X]$ of $\mathcal{A}-$degree $\mathbf{a} \in \mathbb{N} \mathcal{A}$, and define the abstract simplicial complex on the vertex set $M_\mathbf{a}$ $$\nabla_\mathbf{a} = \{ F \subseteq M_\mathbf{a}\ \mid\ \mathrm{gcd}(F) \neq 1\},$$ where $\mathrm{gcd}(F)$ denotes the greatest common divisor of the monomials in $F.$ This simplicial complex was introduced by Eliahou in \cite{Eliahou}, and allows to describe the $\mathcal{A}-$graded minimal free resolution of $\mathbbmss{k}[\mathcal{A}]$ (see for instance \cite{OjVi1, Charalambous}).

We will say that $\mathbf{a} \in \mathbb{N} \mathcal{A}$ is a \textbf{Betti element} if $\nabla_\mathbf{a}$ has more than two connected components. The set of Betti elements of $\mathbb{N} \mathcal{A}$ is denoted $\mathrm{Betti}(\mathbb{N} \mathcal{A})$.

\begin{remark}
The definition of Betti element is equivalent to the one given in \cite{uniquely} (see, e.g. \cite[Proposition 9.7]{libro-fg} or \cite[Theorem 3]{OjVi1}); the main advantage of this definition is that the vertex labels are the factorizations of elements in the semigroup.
\end{remark}

It is well known, that $\mathbf{a}$ is a \textbf{Betti element} of $\mathbb{N} \mathcal{A}$ if and only if $I_\mathcal{A}$ has minimal binomial generator in $\mathcal{A}-$degree $\mathbf{a}$ (see, \cite[Corollary 4]{OjVi1}), or equivalently, there is a pair of factorizations $\mathbf u,\mathbf v$ of $\mathbf{a}$ such that $(\mathbf u,\mathbf v)$ belongs to a minimal presentation of $\mathbb N\mathcal A$ (\cite[Section 2]{uniquely}).

\subsection*{Catenary degree}

Now, consider the \textbf{distance} between two factorizations $\mathbf{u}$ and $\mathbf{v} \in \mathbb N^n$ that is defined as follows 
$$\mathrm{d}(\mathbf{u}, \mathbf{v}) = \max(\deg(X^\mathbf{u}),\deg(X^\mathbf{v})) - \deg(\mathrm{gcd}(X^\mathbf{u}, X^\mathbf{v})).$$ The curious reader may check that $\mathrm{d}$ is actually a metric in the topological sense (see \cite[Proposition 1.2.5]{GHKb} for its basic properties).

Let $N \geq 0,\ \mathbf{a} \in \mathbb{N} \mathcal{A}$ and $\mathbf{u}, \mathbf{v} \in \mathsf Z(\mathbf{a})$. An \textbf{$N-$chain} from $\mathbf{u}$ to $\mathbf{v}$ is a sequence ${\mathbf{u}_0}, \ldots, {\mathbf{u}_k} \in \mathsf Z(\mathbf a)$ such that 
\begin{itemize}
\item $\mathbf{u_0} = \mathbf{u}$ and $\mathbf{u}_k = \mathbf{v}$; 
\item $\mathrm{d}({\mathbf{u}_i}, {\mathbf{u}_{i+1}}) \leq N$, for all $i$.
\end{itemize}
The \textbf{catenary degree} of $\mathbf{a}$, $\mathsf c(\mathbf{a})$, is the minimum $N \in \mathbb{N}$ such that for any two factorizations $\mathbf{u}$ and $\mathbf{v}$ of $\mathbf{a}$ there is an $N-$chain from $\mathbf{u}$ to $\mathbf{v}$. This minimum is always reached, since the set $\mathsf Z(\mathbf a)$ has finitely many elements.

The \textbf{catenary degree} of $\mathbb{N} \mathcal{A}$ is defined by $$\mathsf c(\mathbb{N} \mathcal{A}) = \mathrm{max}\{\mathsf  c(\mathbf{a})\ \mid\ \mathbf{a} \in \mathbb{N} \mathcal{A}\}.$$

From the proof of \cite[Theorem 3.1]{Chapman} it follows that the catenary degree of $\mathbb N\mathcal A$ is reached in one of its Betti elements.

\section{Catenary degree versus Betti degree in a half factorial monoid}

Let $\mathcal{A} = \{\mathbf{a}_1, \ldots, \mathbf{a}_n\} \subseteq \mathbb{Z}^d$. In this section, we assume that $\mathbb N\mathcal A$ is half-factorial, and thus there exists $\omega\in \mathbb Q^d$ such that $A\, \omega^T=(1,\ldots,1)^T$, where $A$ is the matrix whose rows are the elements of $\mathcal A$. 

The proof of the following result is a straightforward consequence of Lemma \ref{longitudes-half-factorial}.

\begin{lemma}\label{distancias-half-factorial}
For $\mathbf u,\mathbf v\in \mathsf Z(\mathbf a)$,   
$$\mathrm{d}(\mathbf{u}, \mathbf{v}) =  \mathbf{a} \cdot \omega - \deg(\mathrm{gcd}(X^\mathbf{u}, X^\mathbf{v})).$$ 
In particular $\mathrm d(\mathbf{u}, \mathbf{v})  \le  \mathbf{a} \cdot \omega$, and the equality holds if and only if $\gcd(X^\mathbf{u}, X^\mathbf{v})=1$ (equivalently $\mathbf u\cdot \mathbf v=0$).
\end{lemma}

Hence, we have that \begin{equation}\label{cat_ecu3}\mathbf{a} \cdot \omega - \max_{\mathbf{u}, \mathbf{v} \in \mathsf{Z}(\mathbf{a})}(\deg(\mathrm{gcd}(X^\mathbf{u}, X^\mathbf{v}))) \leq \mathsf c(\mathbf{a}) \leq \mathbf{a} \cdot \omega,\end{equation} for each $\mathbf{a} \in \mathbb{N} \mathcal{A}.$

We see now that the second inequality becomes and equality precisely when $\mathbf a$ is a Betti element.

\begin{proposition}\label{cat-betti-half-factorial}
Let $\mathbf{b} \in \mathbb{N} \mathcal{A}$. Then $\mathbf{b} \in \mathrm{Betti}(\mathbb{N} \mathcal{A})$ if and only if  $\mathsf c(\mathbf{b}) = \mathbf{b} \cdot \omega$.
\end{proposition}

\begin{proof}
By definition, $\mathbf{b} \in \mathrm{Betti}(\mathbb{N} \mathcal{A}),$ if and only if, there exists $X^\mathbf{u}$ and $X^\mathbf{v} \in M_\mathbf{b}$ lying in different connected components of $\nabla_\mathbf{b}.$ Equivalently, for every chain, $X^{\mathbf{u}_0}, \ldots, X^{\mathbf{u}_k} \in M_\mathbf{b}$ from $X^\mathbf{u}$ to $X^\mathbf{v},$ there exist $j$ such that $\mathrm{gcd}(X^{\mathbf{u}_j}, X^{\mathbf{u}_{j+1}}) = 1;$ that is, $\mathrm{d}({\mathbf{u}_j}, {\mathbf{u}_{j+1}}) = \mathbf{b} \cdot \omega,$ by Lemma \ref{distancias-half-factorial}. Now, since $\mathsf c(\mathbf{b}) \leq  \mathbf{b} \cdot \omega,$  we obtain that the equality must hold. Conversely, if $\mathsf c(\mathbf{b}) =  \mathbf{b} \cdot \omega,$ then there exists $X^\mathbf{u}$ and $X^\mathbf{v} \in M_\mathbf{b}$ such that $\mathrm d(\mathbf u,\mathbf v)=\mathsf c(\mathbf{b})$, and for every chain $X^{\mathbf{u}_0}, \ldots, X^{\mathbf{u}_k}  \in M_\mathbf{b}$ from $X^\mathbf{u}$ to $X^\mathbf{v}$, there exists $j$ such that $\mathrm{d}({\mathbf{u}_j},{\mathbf{u}_{j+1}}) \ge  \mathsf c(\mathbf{b}) =  \mathbf{b} \cdot \omega$. By Lemma \ref{distancias-half-factorial}, this forces  $\mathrm{gcd}(X^{\mathbf{u}_j}, X^{\mathbf{u}_{j+1}}) = 1$. So $X^\mathbf{u}$ and $X^\mathbf{v} \in M_\mathbf{b}$ belong to different connected components of $\nabla_\mathbf{b}$, and we are done.
\end{proof}

We next show that all possible catenary degrees in a half-factorial monoid are attained in its Betti elements.

\begin{theorem}\label{realization-betti-cat}
Let $\mathbb N\mathcal A$ be half-factorial, and let $\mathbf{a} \in \mathbb{N} \mathcal{A}$ with $\#\mathsf Z(a)\ge 2$. There exists $\mathbf{b} \in \mathrm{Betti}(\mathbb{N} \mathcal{A})$ such that $\mathsf c(\mathbf{a}) =\mathsf  c(\mathbf{b}).$
\end{theorem}

\begin{proof}
Let $\omega\in \mathbb Q^d$ be such that $A\, \omega^T=(1,\ldots,1)^T$. 

There exist $X^\mathbf{u}, X^\mathbf{v} \in M_\mathbf{a},$ such that $\mathrm{d}(\mathbf{u}, \mathbf{v}) = \mathsf c(\mathbf{a})$ and, for every chain, $X^{\mathbf{u}_0}, \ldots, X^{\mathbf{u}_k} \in M_\mathbf{a}$ from $X^\mathbf{u}$ to $X^\mathbf{v},$ there exists $j$ with $\mathrm{d}({\mathbf{u}_j}, {\mathbf{u}_{j+1}}) \geq \mathsf c(\mathbf{a});$ thus, for such $j,$ we have $$\mathbf{a} \cdot \omega - \deg(\mathrm{gcd}(X^{\mathbf{u}_j}, X^{\mathbf{u}_{j+1}})) \geq \mathsf c(\mathbf{a}) = \mathbf{a} \cdot \omega - \deg(\mathrm{gcd}(X^\mathbf{u}, X^\mathbf{v})),$$ that is to say, $\deg(\mathrm{gcd}(X^\mathbf{u}, X^\mathbf{v})) \geq \deg(\mathrm{gcd}(X^{\mathbf{u}_j}, X^{\mathbf{u}_{j+1}})).$ In particular, if $\mathrm{gcd}(X^\mathbf{u}, X^\mathbf{v}) \mid \mathrm{gcd}(X^{\mathbf{u}_j}, X^{\mathbf{u}_{j+1}}),$ then they must be equal.

Let $\mathbf{b} = \mathbf{a} - \mathrm{deg}_\mathcal{A}(\mathrm{gcd}(X^\mathbf{u}, X^\mathbf{v})).$ The monomials $X^ {\mathbf{u}'} = X^\mathbf{u}/\mathrm{gcd}(X^\mathbf{u}, X^\mathbf{v})$ and $X^ {\mathbf{v}'} = X^ \mathbf{v}/\mathrm{gcd}(X^\mathbf{u}, X^\mathbf{v})$ have $\mathcal{A}-$degree $\mathbf{b}$ and $$\mathbf{b} \cdot \omega = \mathrm{d}( {\mathbf{u}'}, {\mathbf{v}'}) =  \mathbf{a} \cdot \omega - \mathrm{deg}(\mathrm{gcd}(X^\mathbf{u}, X^\mathbf{v})) = \mathsf c(\mathbf{a}).$$ Now, we prove that $\mathbf{b} \in \mathrm{Betti}(\mathbb{N} \mathcal{A})$. Every chain, $X^{\mathbf{u}'_0}, \ldots, X^{\mathbf{u}'_k} \in  M_\mathbf{b}$ from $X^{\mathbf{u}'}$ to $X^{\mathbf{v}'}$ lifts to a chain, $X^{\mathbf{u}_0}, \ldots, X^{\mathbf{u}_k} \in M_\mathbf{a}$ from $X^\mathbf{u}$ to $X^\mathbf{v}$ (indeed, it to suffices to take $X^{\mathbf{u}_i} = \mathrm{gcd}(X^\mathbf{u}, X^\mathbf{v}) X^{\mathbf{u}'_i},$ for all $i$), and $\mathrm{d}( {\mathbf{u}_j}, {\mathbf{u}_{j+1}}) = \mathrm{d}( {\mathbf{u}'_j}, {\mathbf{u}'_{j+1}})$. By the above arguments, we conclude that there exists $j$ such that $\mathrm{gcd}(X^{\mathbf{u}_j}, X^{\mathbf{u}_{j+1}}) = \mathrm{gcd}(X^\mathbf{u}, X^\mathbf{v}),$ hence $$\mathrm{gcd}(X^{\mathbf{u}'_j}, X^{\mathbf{u}'_{j+1}}) = 1$$ for some $j.$ 

Therefore, it follows that $\mathbf{b}$ is a Betti degree (because $\nabla_\mathbf{b}$ is not connected) and, by Proposition \ref{cat-betti-half-factorial}, $\mathsf c(\mathbf{b}) = \mathbf{b} \cdot \omega = \mathsf c(\mathbf{a}).$ 
\end{proof}

This result does not hold for non half-factorial monoids.

\begin{example}\label{ej-nabla}
Let $\mathcal{A} = \{31, 47, 57\} \subseteq \mathbb{N}$. Then $\mathrm{Betti}(\mathbb{N} \mathcal{A}) = \{171,517,527\}$, and $\mathsf{c}(171) = 5,\ \mathsf{c}(517) = 15$ and $\mathsf{c}(527) = 17.$ However, $\mathsf{c}(564) = 14 \not\in \{5,15,17\}$.

The picture represents $\nabla_{564}.$ The dashed line does not belong to $\nabla_{564}.$ The edges are labeled with the distances between their ends.
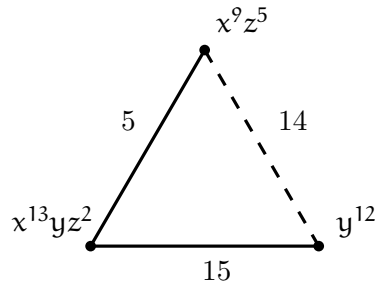
\begin{center}
\begin{tikzpicture}
\draw [line width=1.2pt] (1.5,2.6)-- (0,0);
\draw [line width=1.2pt,dash pattern=on 5pt off 5pt] (3,0)-- (1.5,2.6);
\draw [line width=1.2pt] (0,0)-- (3,0);
\fill (0,0) circle (2pt); \draw (-0.5,0.33) node {$x^{13}yz^2$};
\fill  (3,0) circle (2pt); \draw (3.5,0.33) node {$y^{12}$};
\fill (1.5,2.6) circle (2pt); \draw (2,3) node {$x^9 z^5$};
\draw (0.5,1.66) node {5};
\draw (2.66,1.66) node {14};
\draw (1.66,-0.33) node {15};
\end{tikzpicture}
\end{center}
\end{example}

Recall that the \textbf{total degree} of a polynomial $f \in \mathbbmss{k}[X]$ is the largest of the degrees of its monomials. 
So, the above theorem can be restated as saying that \emph{the catenary degrees of $\mathbb{N} \mathcal{A}$ are the total  degrees of the minimal binomial generators of $I_\mathcal{A}.$}

\begin{corollary}\label{cat-total degree}
The catenary degree of $\mathbb{N} \mathcal{A}$ is the maximum of the total degrees of a minimal system of binomial generators of $I_\mathcal{A}.$
\end{corollary}

The above corollary can also be obtained by adapting the proof of \cite[Theorem 3.1.]{Chapman} to the half-factorial case.

\section{Applications}

Let $\mathcal{A} = \{\mathbf{a}_1, \ldots, \mathbf{a}_n\} \subseteq \mathbb{Z}^d$. In this section, from $\mathbb N \mathcal A,$ we construct two half-factorial monoids. The catenary degree of the first one agrees with the equal catenary degree of the original monoid; while that of the second provides a refinement of the monotone catenary of $\mathbb N \mathcal A$ as an upper bound of its ordinary catenary degree.


For $\mathbf{a} \in \mathbb N \mathcal{A}$ and $i \in \mathsf L (\mathbf{a})$, set $\mathsf{Z}_i(\mathbf{a}) = \{ \mathbf{u} \in \mathsf Z (\mathbf{a})\ : \ \vert \mathbf{u} \vert = i\}.$ The \textbf{equal catenary degree} of $\mathbf{a} \in \mathbb N \mathcal{A}$, $\mathsf{c_{eq}}(\mathbf{a})$, is the minimum $N \in \mathbb{N}$ such that for any $i \in \mathsf L (\mathbf{a})$ and $\mathbf{u}, \mathbf{v} \in \mathsf{Z}_i(\mathbf{a}),$ there is a $N-$ chain from $\mathbf{u}$ to $\mathbf{v}$ in $\mathsf{Z}_i(\mathbf{a})$.

Define $\mathcal A^\mathsf{eq}=\{\mathbf (1,\mathbf a_1),\ldots, (1,\mathbf a_n)\}\subseteq \mathbb N\times \mathbb N^d$. Notice that $(i,\mathbf{a})\in \mathbb{N} \mathcal A^\mathsf{eq}$ if and only if $\mathbf{a} \in \mathbb N \mathcal A$ and $i\in \mathsf L(\mathbf{a})$. Also, observe that $ \mathbb{N} \mathcal A^\mathsf{eq}$ is a half-factorial monoid (just take $\omega = (1,0, \ldots, 0)$). The trick of adding an extra coordinate was already used in \cite{same-length}, where the authors were looking for the existence of factorizations of equal length.

The next result follows easily  from the definitions.

\begin{proposition}\label{Prop ceq-eqc}
$\mathsf{c_{eq}}(\mathbb N \mathcal{A} )=\mathsf c(\mathbb N \mathcal{A} ^\mathsf{eq})$.
\end{proposition}

As consequence of Corollary \ref{cat-total degree}, we obtain the following.

\begin{corollary}
The equal catenary degree of $\mathbb{N} \mathcal{A}$ is the maximum of the total degrees of a minimal system of binomial generators of $I_{\mathcal{A}^{\mathsf eq}}.$
\end{corollary}


Set $\mathcal A^{\textsf{hom}}=\{\mathbf e_0, (1,\mathbf a_1),\ldots, (1,\mathbf a_n)\}\subseteq \mathbb N\times \mathbb N^d$, with $\mathbf e_0= (1,0,\ldots,0)$. As in the previous case, $\mathbb N \mathcal A^{\textsf{hom}}$ is a half-factorial monoid with $\omega = (1,0, \ldots, 0).$ Compare $\mathbf e_0$ with the extra variable  $z$ used in \cite[Section 5.4.6]{char-cmon-2}

Let us see what is the relationship between the factorizations in $\mathbb N \mathcal{A}$ and $\mathbb N \mathcal{A}^{\mathsf{hom}}$.

\begin{lemma}\label{fac-homogeneo}
$\mathsf Z \big( (i, \mathbf{a}) \big)= \big\{ (j,\mathbf{u}) \in \mathbb{N} \times  \mathsf Z (\mathbf{a}) ~|~ j = i-|\mathbf{u}| \big\}$.
\end{lemma}

\begin{proof}
Let $(u_0, \ldots, u_n) \in \mathsf Z \big( (i, \mathbf{a}) \big)$. Then $u_0\mathbf e_0+u_1(1,\mathbf a_1)+ \cdots + u_n(1,\mathbf a_n)=(i,\mathbf a)$. This implies that $\mathbf a =u_1\mathbf a_1+\cdots+ u_n \mathbf a_n$ and $i=u_0+u_1+\cdots +u_n$. Take $j=u_0$ and $\mathbf{u}=(u_1,\ldots, u_n)$. The other inclusion is also straightforward.
\end{proof}

Now, we see that the distances of factorizations of an element in $\mathbb N \mathcal{A}^{\mathsf{hom}}$ are ruled by the factorizations of the corresponding one in $\mathbb N \mathcal{A}$.

\begin{lemma}\label{distancias}
 Let $(i,\mathbf{a})\in \mathbb N \mathcal{A}^{\textsf{hom}}$, and let $(j_\mathbf{u},\mathbf{u}), (j_\mathbf{v},\mathbf{v})\in \mathsf Z \big( (i, \mathbf{a}) \big)$. Then $\mathrm d\big((j_\mathbf{u},\mathbf{u}),(j_\mathbf{v},\mathbf{v})\big)=\mathrm d(\mathbf{u},\mathbf{v})$.
\end{lemma}

\begin{proof}
Notice that $i=|\mathbf{u}|+j_\mathbf{u}=|\mathbf{v}|+j_\mathbf{v}$. Assume without loss of generality that $|\mathbf v|\ge |\mathbf u|$. Set $X^\mathbf w=\gcd(X^\mathbf u, X^\mathbf v)$. Then $\gcd(X_0^{j_\mathbf u} X^\mathbf u, X_0^{j_\mathbf v} X^\mathbf v))=X_0^{j_\mathbf v} X^\mathbf w$, and $\mathrm d((j_\mathbf u,\mathbf u),(j_\mathbf v,\mathbf v)) \stackrel{\text{Lemma \ref{distancias-half-factorial}}}{=} i-(j_\mathbf v+|\mathbf w|) =|\mathbf v|-|\mathbf w|=\mathrm d(\mathbf u,\mathbf v)$.
\end{proof}

The \textbf{homogeneous catenary degree}, $\mathsf{c_{hom}}(\mathbf{a})$, of an element $\mathbf{a} \in \mathbb N \mathcal A$ is the least $N\in \mathbb N$ such that for any $\mathbf{u}, \mathbf{v} \in \mathsf Z(\mathbf a)$ there exists $N-$ chain from $\mathbf{u}$ to $\mathbf{v}$ in $\mathsf Z (\mathbf a) \cap \big\{ \mathbf w : \vert \mathbf w \vert \le \max\{ \vert \mathbf{u} \vert, \vert \mathbf{v} \vert \} \big\}$. If no $N \in \mathbb N$ do exist, we define $\mathsf{c_{hom}}(\mathbf{a}) = \infty.$

The homogeneous catenary degree of $\mathbb N \mathcal A$ is the supremum of all homogeneous catenary degrees of its elements. This definition was inspired by the following result.

\begin{proposition}\label{Prop chom-homc}
 $\mathsf {c_{hom}}(\mathbb N \mathcal A)=\mathsf c(\mathbb N \mathcal{A}^{\textsf{hom}})$.
\end{proposition}

\begin{proof}
Let $\mathbf{u}, \mathbf{v} \in \mathsf Z (\mathbf a),$ for some $\mathbf a \in \mathbb N \mathcal A.$ Assume without loss of generality that $j_\mathbf{u} = |\mathbf u|\le |\mathbf v|=j_\mathbf v$. Then $(j_\mathbf v - j_\mathbf u,\mathbf u), (0,\mathbf v)$ are factorizations of $(j_\mathbf v,\mathbf a)$. There exists a $\mathsf c(\mathbb N \mathcal{A}^{\textsf{hom}})$-chain  $(j_1,\mathbf{w}_1),\ldots, (j_t,\mathbf{w}_t)$ joining them. As $j_k=j_\mathbf{v}-|{\mathbf w_k}|$, we have that $|\mathbf{w}_k|\le |\mathbf v|$, and thus $\mathbf{w}_1,\ldots,\mathbf{w}_t$ is a $\mathsf c(\mathbb N \mathcal{A}^{\textsf{hom}})$-chain joining $\mathbf u$ and $\mathbf v$ with $|\mathbf{w}_k|\le \max\{|\mathbf u|,|\mathbf v|\}$. This proves $\mathsf {c_{hom}}(\mathbb N \mathcal A)\le\mathsf c(\mathbb N \mathcal{A}^{\textsf{hom}})$. 

Conversely, let $(j_\mathbf{u}, \mathbf{u}), (j_\mathbf{v}, \mathbf{v})$ be factorizations of $(i,\mathbf a)\in \mathbb N \mathcal{A}^{\textsf{hom}}$. In view of Lemma \ref{fac-homogeneo}, $j_\mathbf{u}+|\mathbf u|=j_\mathbf{v}+|\mathbf v|=i$. Assume without loss of generality that $|\mathbf u|\le |\mathbf v|$. Let $\mathbf{w}_1, \ldots, \mathbf{w}_t$ be a $\mathsf {c_{hom}}( \mathbb N \mathcal{A})$-chain from $\mathbf{u}$ to $\mathbf{v}.$ By definition, $|\mathbf{w}_k|\le |\mathbf v|\le i$. Set $j_k=i-|\mathbf w_k|$. Then $(j_1,\mathbf{w}_1),\ldots, (j_t,\mathbf{w}_t)$ is a $\mathsf {c_{hom}}( \mathbb N \mathcal{A})$-chain joining $(j_\mathbf{u}, \mathbf{u}), (j_\mathbf{v}, \mathbf{v})$. Thus, $\mathsf c(\mathbb N \mathcal{A}^{\textsf{hom}}) \le \mathsf {c_{hom}}(\mathbb N \mathcal A),$ and this completes the proof.
\end{proof}

As consequence of Corollary \ref{cat-total degree}, we obtain the following.

\begin{corollary}\label{cat-betti-hom}
The homogeneous catenary degree of $\mathbb{N} \mathcal{A}$ is the maximum of the total degrees of a minimal system of binomial generators of $I_{\mathcal{A}^{\textsf{hom}}}.$
\end{corollary}

We prove that our new catenary degree is an upper bound for the usual catenary degree.

\begin{proposition}\label{cadenas-homogeneo}
$\mathsf c(\mathbb N \mathcal{A})\le \mathsf {c_{hom}}(\mathbb N \mathcal{A})$.
\end{proposition}

\begin{proof}
Let $\mathbf a \in \mathbb N \mathcal{A},$ and let $\mathbf{u}, \mathbf{v} \in \mathsf Z (\mathbf a)$ with $| \mathbf{u} | \leq | \mathbf{v} |$. We show that there exists a $\mathsf c(\mathbb N \mathcal{A}^{\textsf{hom}})$-chain joining $\mathbf{u}$ and $\mathbf{v}$. Set $j_\mathbf{u} = | \mathbf{u} | \leq | \mathbf{v} | =j_\mathbf{v}$. Then $(j_\mathbf{v} - j_\mathbf{u}, \mathbf{u})$ and $(0,\mathbf{v}) \in \mathsf Z (j_\mathbf{v}, \mathbf{a})$. From the definition of homogeneous catenary degree, there exists a $\mathsf c(\mathbb N \mathcal{A}^{\textsf{hom}})-$chain $(j_1, \mathbf{w}_1), \ldots, (j_t, \mathbf{w}_t)$ of factorizations of  $j_\mathbf{v}, \mathbf{a}$ from $(j_\mathbf{v} - j_\mathbf{u}, \mathbf{u})$ to $(0,\mathbf{v})$, and $\mathsf d((j_k,\mathbf{w}_k),(j_{k+1},\mathbf{w}_{k+1}))\le \mathsf c(\mathbb N \mathcal{A}^{\textsf{hom}})$ From Lemma \ref{distancias}, $\mathsf d((j_k,\mathbf{w}_k),(j_{k+1},\mathbf{w}_{k+1}))=\mathsf d(\mathbf{w}_k,\mathbf{w}_{k+1})$, whence $\mathbf{w}_1,\ldots,\mathbf{w}_t$ is a $\mathsf c(\mathbb N \mathcal{A}^{\textsf{hom}})$-chain joining $\mathbf{u}$ and $\mathbf{v}$.
\end{proof}

The catenary degree might be strictly smaller than the homogeneous catenary degree.

\begin{example}
Let $\mathcal{A} = \{10,11,14,19\}$. One can check that $\mathsf c(\mathbb N \mathcal A) = 4.$ Since a minimal system of binomial generators of $I_{\mathcal{A}^{\mathsf{hom}}} \subseteq \mathbbmss{k}[X_0, \ldots, X_4]$ is $\{X_2X_3^2-X_1^2X_4, X_1X_3^2-X_0X_4^2, X_2^3-X_0X_3X_4, X_1^3-X_0X_2X_4, X_1^2X_2^2-X_0X_3^3, X_3^5-X_1X_2^2X_4^2\}$, we may conclude, by Corollary \ref{cat-betti-hom}, that $\mathsf {c_{hom}}(\mathbb N \mathcal{A}) = 5.$
\end{example}

We now compare our new catenary degree with the widely studied monotone catenary degree. Recall that the \textbf{monotone catenary degree}, $\mathsf{c_{hom}}(\mathbf{a})$, of an element $\mathbf{a} \in \mathbb N \mathcal A$ is the least $N\in \mathbb N$ such that for any two factorizations $\mathbf{u}$ and $\mathbf{v}$ of $\mathbf{a}$ with $\vert \mathbf{u} \vert \leq \vert \mathbf{v} \vert$ there is an $N-$chain $\mathbf{u} = \mathbf{u}_0, \ldots, \mathbf{u}_k = \mathbf{v}$ with $\vert \mathbf{u}_0 \vert \le \cdots \le \vert \mathbf{u}_k \vert$.

\begin{proposition}
 $\mathsf {c_{hom}}(\mathbb N \mathcal{A}) \le \mathsf{c_{mon}}(\mathbb N \mathcal{A})$.
\end{proposition}

\begin{proof}
Let $(i, \mathbf{a}) \in \mathbb N \mathcal{A}^{\textsf{hom}}$, and let $(j_\mathbf{u}, \mathbf{u}),\ (j_\mathbf{v},\mathbf{v}) \in \mathsf Z ((i, \mathbf{a}))$. Assume for instance that $i-j_\mathbf{u}=|\mathbf{u}| \le |\mathbf{v}| = i - j_\mathbf{v}$. From the definition of $\mathsf{c_{mon}}(N \mathcal{A})$, there exist $\mathbf{w}_1, \ldots, \mathbf{w}_t \in \mathsf Z (\mathbf a)$ with $\mathbf{w}_1=\mathbf{u}$, $\mathbf{w}_t=\mathbf{v}$, $\mathrm d(\mathbf{w}_k,\mathbf{w}_{k+1})\le \mathsf{c_{mon}}(\mathbb N \mathcal A)$ and $|\mathbf{w}_k|\le |\mathbf{w}_{k+1}|$. Set $j_k=i-|\mathbf{w}_k|$. Then $(j_1,\mathbf{w}_1), \ldots, (j_t,\mathbf{w}_t)$ is a $\mathsf{c_{mon}}(\mathbb N \mathcal A)$-chain joining $(j_\mathbf{u},\mathbf{u})$ and $(j_\mathbf{v},\mathbf{v})$. Thus $\mathsf c(\mathbb N \mathcal{A}^{\textsf{hom}})\le \mathsf{c_{mon}}(\mathbb N \mathcal A)$.
\end{proof}

In some cases the homogeneous catenary degree is sharper than the monotone catenary degree.

\begin{example}
Let $\mathcal{A} = \{11,19,32\}.$ Then $\mathsf c (\mathbb N \mathcal{A}) =  \mathsf {c_{hom}}(\mathbb N \mathcal{A}) = 11 < \mathsf{c_{eq}}(\mathbb N \mathcal{A}) = \mathsf{c_{mon}}(\mathbb N \mathcal{A}) = 21$.
\end{example}

In spite of the above example, the equal catenary degree may be smaller than the homogeneous catenary degree.

\begin{example}
For $\mathcal{A} = \{11,19,23\}$, $\mathsf c(\mathbb N\mathcal A)=\mathsf{c_{eq}}(\mathbb N \mathcal{A}) = 3 < \mathsf {c_{hom}}(\mathbb N \mathcal{A}) =\mathsf{ c_{mon}} (\mathbb N\mathcal A)= 9$.
\end{example}

\section{Other invariants}

Let as above $\mathcal A= \{\mathbf{a}_1, \ldots, \mathbf{a}_n\} \subseteq \mathbb{Z}^d$. We assume now that 
$\mathcal A$ is a minimal system of generators of $\mathbb N \mathcal A$.

Another invariant related with distances of factorizations is the tame degree: the \textbf{tame degree} of $\mathbf a \in \mathbb N\mathcal A$, $\mathsf t(\mathbf a)$, is the minimum of all $N\in \mathbb N$ such that for all $\mathbf{u}\in {\mathsf Z}(\mathbf{a})$ and every minimal generator $\mathbf{a}_i$ such that $\mathbf{a}-\mathbf{a}_i\in \mathbb N\mathcal A$, there exists $\mathbf{u}' =(u_1',\ldots, u_n') \in {\mathsf Z}(\mathbf{a})$ such that $u'_i\neq 0$ and $\mathsf{d}(\mathbf{u},\mathbf{u}')\leq N$. The tame degree of $\mathbb N\mathcal A$, $\mathsf t(\mathbb N\mathcal A)$, is the supremum of all tame degrees of its elements.

\begin{proposition}
 $\mathsf t(\mathbb N \mathcal{A})\le \mathsf t(\mathbb N \mathcal{A}^{\textsf{hom}})$.
\end{proposition}
\begin{proof}
 Let $\mathbf a\in \mathbb N \mathcal{A}$ and $i\in\{1,\ldots,n\}$ be such that $\mathbf a'=\mathbf a-\mathbf a_i\in \mathbb N \mathcal{A}$. Assume that there exists $\mathbf u=(u_1,\ldots,u_n)\in \mathsf Z (\mathbf a)$ with $u_i=0$. Let $j=\max \mathsf L(\mathbf a)$, $j'=\max \mathsf L(\mathbf a')$, and $l_\mathbf u=j-|\mathbf u|$. Let $\mathbf v\in \mathsf Z(\mathbf \mathbf a')$ be such that $|\mathbf v|=j'$. As $\mathbf v+\mathbf e_i\in \mathsf Z(\mathbf a)$, we deduce that $j'+1\le j$. Then $(j,\mathbf a)$ and $(j-1,\mathbf a-\mathbf a_i)=(j,\mathbf a)-(1,\mathbf a_i)\in \mathbb N \mathcal{A}^{\textsf{hom}}$, and $(l_\mathbf u,\mathbf u)\in \mathsf Z \big( (j,\mathbf a) \big)$. So by definition of $\mathsf t(\mathbb N \mathcal{A}^{\textsf{hom}})$, there exists $(l_{\mathbf w},\mathbf w)\in \mathsf Z \big( (j,\mathbf  a) \big)$ with $\mathbf w \cdot \mathbf e_i\neq 0$ and $\mathrm d\big( (l_\mathbf u,\mathbf u),(l_\mathbf w,\mathbf w) \big)\le \mathsf t(\mathbb N \mathcal{A}^{\textsf{hom}})$. From Lemma \ref{distancias}, we deduce that $\mathrm d(\mathbf u,\mathbf w)\le \mathsf t(\mathbb N \mathcal{A}^{\textsf{hom}})$. This proves that $\mathsf t(\mathbb N \mathcal{A})\le \mathsf t(\mathbb N \mathcal{A}^{\textsf{hom}})$. 
\end{proof}

\begin{notation}
Given $\mathbf a, \mathbf a'$ in $\mathbb Z^d$, we write $\mathbf a \preceq \mathbf a'$ if $\mathbf a'-\mathbf a\in \mathbb N\mathcal A$, and given $\mathbf u, \mathbf u'$ in $\mathbb Z^n,$ we write $\mathbf{u} \le \mathbf{u}'$ if $\mathbf{u}' - \mathbf{u} \in \mathbb N^n$.
\end{notation}

\begin{lemma}\label{minimales-soportes}
Let $\mathbf u \in \pi^{-1}(\mathbf a_i +\mathbb N \mathcal{A}) \setminus\{\mathbf e_i\}$ be minimal (with respect to $\le$) in $\pi^{-1}(\mathbf a_i +\mathbb N \mathcal{A})$, for some $i\in \{1,\ldots,n\}$, and let $\mathbf a=\pi(\mathbf u)$. Then $\mathbf u\cdot \mathbf v=0$ for all $\mathbf v=(v_1,\ldots,v_n)\in \mathsf Z(\mathbf a)$ such that $v_i\neq 0$.
\end{lemma}

\begin{proof}
Observe that if $\mathbf u=(u_1,\ldots,u_n)$, then $u_i=0$. Notice also that since $\mathbf a\in \mathbf a_i+\mathbb N\mathcal A$, there exists $\mathbf v=(v_1,\ldots,v_n)\in \mathsf Z(\mathbf a)$ such that $v_i\neq 0$. Assume that $\mathbf u\cdot \mathbf v\neq 0$. As $u_i=0$, this means that there exists $j\in\{1,\ldots,n\}\setminus\{i\}$ with $u_j\neq 0\neq v_j$. But then $\pi(\mathbf v)=\mathbf a_i+\mathbf a_j+\mathbf a'$ for some $\mathbf a'\in \mathbb N\mathcal A$, and consequently $\pi(\mathbf u-\mathbf e_j)=\pi(\mathbf v-\mathbf e_j)\in \mathbf a_i+\mathbb N\mathcal A$, contradicting the minimality of $\mathbf u$.  
\end{proof}

There is still another non-unique factorization invariant that apparently has nothing to do with distances, and measures how far an element is from being a prime. 

The \textbf{$\omega$-primality} of $\mathbf a$, $\omega(\mathbf a)$, is the least positive integer such that whenever $\mathbf s_1+\cdots+\mathbf s_k-\mathbf a\in \mathbb \mathcal A$ for some $\mathbf s_1,\ldots,\mathbf  s_k\in \mathbb N \mathcal{A}$, then $\mathbf s_{i_1}+\cdots+\mathbf s_{i_{\omega(\mathbf a)}}-\mathbf a\in \mathbb \mathcal A$ for some  $\{i_1,\ldots,i_{\omega(\mathbf a)}\}\subseteq \{1,\ldots, k\}$.  We can restrict the search to sums of the form $\mathbf s_1+\cdots+\mathbf s_k$, with $\mathbf s_1,\ldots,\mathbf s_k\in \{\mathbf a_1,\ldots,\mathbf a_n\}$ (see \cite[Lemma 3.2]{omega}). In particular, $\omega(\mathbf a)=1$ means that $\mathbf a$ is prime.

Given $\mathbf a\in \mathbb N \mathcal{A}$, $\omega(\mathbf a)$ can be computed in the following form (\cite[Proposition 3.3]{omega})
\begin{equation}\label{calculo-omega}
\omega(\mathbf a)=\sup\big\{ |\mathbf u| : \mathbf u\ \text{minimal in}\ \pi^{-1}(\mathbf a +\mathbb N \mathcal{A})  \big\}. 
\end{equation}
In our setting, thanks to Dickson's lemma, this supremum turns out to be a maximum.

The $\omega$-primality of $\mathbb N \mathcal{A}$ is defined as $\omega(\mathbb N \mathcal{A})=\max_{i\in \{1,\ldots,n\}}\{\omega(\mathbf a_i)\}$. In the half-factorial case, both tame degree and $\omega$-primality coincide. 

\begin{proposition}
Assume that $\mathbb N \mathcal{A}$ is half-factorial. Then  \[\omega(\mathbb N \mathcal{A})=\mathsf t(\mathbb N \mathcal{A}).\]
\end{proposition}

\begin{proof}
It is well known that $\omega(\mathbb N \mathcal{A})\le \mathsf t(\mathbb N \mathcal{A})$ (\cite[Theorem 3.6]{local-tame}). So we only have to prove the other inequality. Let $\mathbf a\in \mathbb N \mathcal{A}$ be minimal with respect to $\le_{\mathbb N \mathcal{A}}$ fulfilling that $\mathsf t(\mathbf a)=\mathsf t(\mathbb N \mathcal{A})$. Then according to \cite[Lemma 5.4]{omega}, there exists $\mathbf u,\mathbf v\in \mathsf Z(s)$, such that $\mathsf t(\mathbf a)=\mathrm d(\mathbf u,\mathbf v)$ with    $\mathbf u$ minimal (with respect to $\le$) in $\pi^{-1}(\mathbf a_i +\mathbb N \mathcal{A})$, $\mathbf u\cdot \mathbf e_i=0$ and $\mathbf v\cdot \mathbf e_i\neq 0$. In light of Lemma \ref{minimales-soportes},  $\mathbf u\cdot \mathbf v=0$, whence $\mathrm d(\mathbf u ,\mathbf v)=\max\{|\mathbf u|,|\mathbf v|\}$. As $\mathbb N \mathcal{A}$ is half factorial, we obtain $\max\{|\mathbf u|,|\mathbf v|\}=|\mathbf u|=|\mathbf v|$. Hence $\mathsf t(\mathbf a)=|\mathbf u|$. From (\ref{calculo-omega}), we conclude that $|\mathbf u|\le \omega(\mathbf a_i)\le \omega(\mathbb N \mathcal{A})$.
\end{proof}

In a private communication, A. Geroldinger told us that this last result can be also derived from the results appearing in \cite[Section 3]{local-tame}.

\begin{example}
 It is well known that $\mathsf c(\mathbb N \mathcal A)\le \mathsf \omega(\mathbb N \mathcal A)$ (see \cite[Section 3]{GK}). In the half-factorial case, this inequality might be strict. For instance, if we take $\mathcal A=\{(1,0),(1,3),(1,5),(1,7)\}$, then $\mathsf c(\mathbb N \mathcal A)=4 < 7=\mathsf \omega(\mathbb N \mathcal A)$. 
\end{example}



\begin{thebibliography}{99}

\bibitem{omega} \textsc{Blanco, V.; Garc\'{\i}a-S\'anchez, {P.A.}; Geroldinger, A.} \emph{Semigroup theoretical characterizations of arithmetical invariants with applications to numerical monoids and Krull monoids}. Illinois J. Math., to appear.

\bibitem{Charalambous} \textsc{Charalambous, H.; Thoma, A.} \emph{On simple A-multigraded minimal resolutions}. Combinatorial aspects of commutative algebra, 33--44, Contemp. Math., \textbf{502}, Amer. Math. Soc., Providence, RI, 2009.

\bibitem{same-length} \textsc{Chapman, S. T.; Garc\'{\i}a-S\'anchez, P. A.; Llena, D.; Marshall, J.}, \emph{Elements in a numerical semigroup with factorizations with the same length}, Canadian Math. Bull. \textbf{54} (2011), 39-43.

\bibitem{Chapman} \textsc{Chapman, S. T.; Garc\'{\i}a-S\'anchez, P. A.; Llena, D.; Ponomarenko, V.; Rosales, J. C.} \emph{The catenary and tame degree in finitely generated commutative cancellative monoids}. Manuscripta Math. \textbf{120} (2006), no. 3, 253-264.

\bibitem{numericalsgps}  \textsc{Delgado, M.; Garc{\'i}a-S{\'a}nchez, {P.A.};Morais, J.}
  \emph{``numericalsgps'': a {\sf {g}{a}{p}} package on numerical semigroups},\\
  \url{http://www.gap-system.org/Packages/numericalsgps.html}.

\bibitem{Eliahou} \textsc{Eliahou, S.}, \emph{Courbes monomiales et alg\`ebre de Rees symboliquem}, Ph.D. Thesis, Universit\'e of Gen\'eve, 1983 (in French). 

\bibitem{uniquely}  \textsc{Garc{\'i}a-S{\'a}nchez, {P.A.}; Ojeda, I.}  \emph{Uniquely presented finitely generated commutative monoids}, Pacific J. Math. \textbf{248} (2010), 91--105.

\bibitem{local-tame} \textsc{Geroldinger, A.; Hassler, W.} \emph{Local tameness of v-noetherian monoids}. J. Pure Appl. Algebra \textbf{212} (2008), 1509-1524.

\bibitem{GK} \textsc{Geroldinger, A.; Kainrath, F.} \emph{On the arithmetic of tame monoids with
  applications to {K}rull monoids and {M}ori domains}, J. Pure Appl. Algebra
  \textbf{214} (2010), 2199 -- 2218.

\bibitem{GHKb} 
\textsc{Geroldinger, A.; {Halter-Koch}, F.} \emph{Non-{U}nique {F}actorizations. {A}lgebraic, {C}ombinatorial and {A}nalytic {T}heory}, Pure and Applied  Mathematics, vol. 278, Chapman \& Hall/CRC, 2006.

\bibitem{herzog} \textsc{Herzog, H.} \emph{Generators and relations of abelian semigroups and semigroup rings.} Manuscripta Math., {\bf 3} (1970), 175-193.

\bibitem{OjVi1} \textsc{Ojeda, I.; Vigneron-Tenorio, A.} \emph{Simplicial complexes and minimal free resolution of monomial algebras}. J. Pure Appl. Algebra \textbf{214} (2010), no. 6, 850--861.

\bibitem{char-cmon-2} \textsc{Philipp, A.} \emph{A characterization of arithmetical invariants by the monoid of relations}, Semigroup Forum {\bf 81} (2010), 424-434.

\bibitem{libro-fg} \textsc{Rosales, J. C.; Garc\'{\i}a-S\'anchez, P. A.},  \emph{Finitely generated commutative monoids}, Nova Science Publishers, Inc., New York, 1999.

\bibitem{Sturmfels} \textsc{Sturmfels, B.}, \emph{Gr\"obner bases and convex polytopes}, volume 8 of University Lecture Series, American Mathematical Society, Providence, RI, 1996.

\end{thebibliography}
\end{document}